\documentclass[12pt, reqno]{amsart}
\usepackage{amsmath}
\usepackage{amssymb}
\usepackage{amsfonts}
\usepackage{amsthm}
\usepackage{amsbsy}
\usepackage{geometry}
\usepackage{enumitem}
\usepackage{mathrsfs}
\usepackage[utf8]{inputenc}
\usepackage{mathtools}
\usepackage{hyperref}
\usepackage[super]{nth}

\usepackage{comment}
\usepackage[colorinlistoftodos,textsize=footnotesize,textwidth=15ex]{todonotes}

\newtheorem{thmintro}{Theorem}

\theoremstyle{definition}
\newtheorem{rkintro}{Remark}

\theoremstyle{plain}
\newtheorem{theorem}{Theorem}[section]

\newtheorem{lemma}[theorem]{Lemma}
\newtheorem{corollary}[theorem]{Corollary}

\theoremstyle{definition}

\newtheorem{example}[theorem]{Example}
\newtheorem{remark}[theorem]{Remark}

\numberwithin{equation}{section}

\newcommand{\triv}{\{1\}}
\renewcommand{\phi}{\varphi}
\renewcommand{\epsilon}{\varepsilon}
\newcommand{\NN}{\mathbb{N}}

\newcommand{\RR}{\mathbb{R}}
\newcommand{\FF}{\mathbb{F}}

\newcommand{\mc}[1]{\mathcal{#1}}

\newcommand{\lin}{\mathrm{lin}}
\newcommand{\halfspace}{\operatorname{hs}}
\newcommand{\co}{\colon\thinspace}

\newcommand{\gen}[1]{\left\langle #1 \right\rangle}

\newcommand{\inv}{^{-1}}

\DeclareMathOperator{\GL}{GL}

\begin{document}
	
\title{Non-linearizable Root Group Data}

\author{Sebastian \textsc{Bischof}}
\email{sebastian.bischof@math.uni-paderborn.de}
\address{Universität Paderborn, Institut für Mathematik, 33098 Paderborn, Germany}

\subjclass[2020]{20E42, 20F55}

\keywords{RGD~systems, commutation relations, root groups, root systems}

\thanks{This work was funded by the Deutsche Forschungsgemeinschaft (DFG, German Research Foundation) -- 536272274 (Walter Benjamin Programme).}

\begin{abstract}
	An RGD~system $\mathcal{D}$ is called \emph{linear w.r.t.\ a root basis $\mathcal{B}$} if the commutation relations between the root groups of $\mathcal{D}$ are `linear' in a certain sense. Moreover, $\mathcal{D}$ is called \emph{linearizable}, if there exists a root basis $\mathcal{B}$ such that $\mathcal{D}$ is linear w.r.t.\ $\mathcal{B}$. For many examples of RGD~systems it is easy to see that they are linear w.r.t.\ a concrete root basis. To the best of our knowledge, it was unclear whether RGD~systems exist which are not linearizable.
	
	In this article we show that there exist uncountably many RGD~systems which are not linearizable. In particular, we provide the first explicit example of such an RGD~system. This expands the quote from R\'{e}my that axiom (RGD$1$)$_{\mathrm{lin}}$ is not only a strengthening of axiom (RGD$1$), but is in fact stronger than it. We show that non-linearizability appears in examples of universal type, and also in examples of $2$-spherical type. For the examples of universal type we construct an uncountable family of non-linearizable RGD~systems, and for the examples of $2$-spherical type we show that the RGD~systems of type $(4, 4, 4)$ recently constructed by the author provide uncountably many non-linearizable RGD~systems.
\end{abstract}

\maketitle

\section{Introduction}

Let $(W, S)$ be a Coxeter system and denote for all $s, t \in S$ the order of $st$ in $W$ by $m_{st}$. We can associate a root system $\Phi := \Phi(W, S)$ to $(W, S)$ on which $W$ acts in a natural way. Any root is either positive or negative, and we denote by $\Phi_{\pm}$ the set of all positive/negative roots. A pair of roots $\{ \alpha, \beta \} \subseteq \Phi$ is \emph{prenilpotent} if there exist $w, w' \in W$ such that
\[ \{ w.\alpha, w.\beta \} \subseteq \Phi_+ \qquad\text{and}\qquad \{ w'.\alpha, w'.\beta \} \subseteq \Phi_-. \]
For any prenilpotent pair $\{\alpha, \beta\} \subseteq \Phi$ one can define a \emph{geometric} interval as follows:
\[ [\alpha, \beta] := \bigcap_{\substack{w\in W \\ \epsilon \in \{+, -\}}} \{ \gamma \in \Phi \mid \{ w.\alpha, w.\beta \} \subseteq \Phi_{\epsilon} \Rightarrow w.\gamma \in \Phi_{\epsilon}  \}. \]
In addition, we define $(\alpha, \beta) := [\alpha, \beta] \setminus \{\alpha, \beta\}$. Roots can be viewed as half-spaces in the Cayley graph of $(W, S)$ -- this root system will be denoted by $\Phi(W, S)$ -- or as vectors in a real vector space by using \emph{reduced root bases $\mc{B}$ associated with $(W, S)$} -- one can associate a natural root system $\Phi(\mc{B})$ to the reduced root basis $\mc{B}$. In the latter case, $W$ acts (faithfully) on the real vector space, and there is a $W$-equivariant bijection $\phi_{\mc{B}}\co \Phi(W, S) \to \Phi(\mc{B})$ between $\Phi(W, S)$ and $\Phi(\mc{B})$. Typical examples of reduced root bases are $\mc{B}(W, S)$ -- the root basis provided by the canonical linear representation of $(W, S)$ -- and $\mc{B}(A)$ -- the root basis provided by the set of real roots of a generalized Cartan matrix $A$. Let $\mc{B}$ be a reduced root basis which is associated with $(W, S)$, and let $\phi\co \Phi(W, S) \to \Phi(\mc{B})$ denote the $W$-equivariant bijection. Then one can also define an \emph{algebraic} interval w.r.t.\ $\mc{B}$ of the prenilpotent pair $\{\alpha, \beta\} \subseteq \Phi(W, S)$ by
\[ [\alpha, \beta]_{\lin}^{\mc{B}} := \phi\inv\left( \RR_{\geq 0} \phi(\alpha) + \RR_{\geq 0} \phi(\beta) \right) \cap \Phi(W, S) \quad\text{as well as}\quad (\alpha, \beta)_{\lin}^{\mc{B}} := [\alpha, \beta]_{\lin}^{\mc{B}} \setminus \{\alpha, \beta\}. \]
It is not hard to see that $[\alpha, \beta]_{\lin}^{\mc{B}}$ is always contained in $[\alpha, \beta]$. It was observed by Tits that this inclusion may be proper (see \cite[(I.1.3) in the academic year $1988/89$]{Ti73-00} and also \cite[$\S$5.4.2]{Re03}). At this point we should note that it is not clear to us whether the set $[\alpha, \beta]_{\lin}^{\mc{B}}$ is independent of the root basis $\mc{B}$ (see also Remark~\ref{Remark: independent intervals}).

In \cite{Ti92}, Tits introduced RGD~systems in order to study Kac--Moody groups over fields. An \emph{RGD~system of type $(W, S)$} is a pair $\mc{D} = (G, (U_{\alpha})_{\alpha \in \Phi(W, S)})$ consisting of a group $G$ and a family of subgroups $U_{\alpha}$ called \emph{root groups} indexed by the set of roots $\Phi(W, S)$ satisfying some axioms. The most important axiom makes an assumption about the commutation relations between root groups corresponding to prenilpotent pairs $\{\alpha, \beta\} \subseteq \Phi(W, S)$ with $\alpha \neq \beta$:
\begin{equation}\label{Equation: RGD1}
	[U_{\alpha}, U_{\beta}] \leq \langle U_{\gamma} \mid \gamma \in (\alpha, \beta) \rangle. \tag*{(RGD1)}
\end{equation}
The RGD~system $\mc{D}$ is called \emph{linear w.r.t.\ a reduced root basis $\mc{B}$} if for each prenilpotent pair $\{\alpha, \beta\} \subseteq \Phi(W, S)$ with $\alpha \neq \beta$ the following holds:\footnote{In \cite{Re03}, it was already mentioned by Rémy that \ref{Equation: RGD1lin} is a strengthening of \ref{Equation: RGD1}. To see that \ref{Equation: RGD1lin} is indeed stronger than \ref{Equation: RGD1}, one has to find an explicit example of an RGD~system which does not satisfy \ref{Equation: RGD1lin}. This is the focus of the present article.}
\begin{equation}\label{Equation: RGD1lin}
	[U_{\alpha}, U_{\beta}] \leq \langle U_{\gamma} \mid \gamma \in (\alpha, \beta)_{\lin}^{\mc{B}} \rangle. \tag*{(RGD$1$)$_{\lin}$}
\end{equation}
The RGD~system $\mc{D}$ is said to be \emph{linearizable} if there exists a reduced root basis $\mc{B}$ such that $\mc{D}$ is linear w.r.t.\ $\mc{B}$. Rémy established in \cite[6.2.2]{Re03} the existence of a Levi-decomposition for parabolic subgroups of linear RGD~systems (w.r.t.\ a reduced root basis) generalizing results from `classical' RGD~systems (i.e.\ RGD~systems of \emph{spherical} type, that is, $W$ is finite). Hence, the structure of linearizable RGD~systems is closer to the structure of classical examples. This is not surprising, as for those, the commutation relations do not just have a geometric interpretation, but also an algebraic one. However, up until now and to the best of our knowledge it was not clear whether non-linearizable RGD~systems exist.

For a root $\alpha \in \Phi(W, S)$, we denote by $r_{\alpha}$ the unique reflection which interchanges $\alpha$ and its opposite root. Let us mention that if the order $o(r_{\alpha} r_{\beta})$ is finite for the prenilpotent pair $\{\alpha, \beta\}$, then $[\alpha, \beta]_{\lin}^{\mc{B}} = [\alpha, \beta]$ for any reduced root basis $\mc{B}$ associated with $(W, S)$ (see Lemma~\ref{Lemma: rank 2 intervals coindice}). In particular, an RGD~system is linearizable if and only if there exists a reduced root basis $\mc{B}$ such that \ref{Equation: RGD1lin} holds for all prenilpotent pairs $\{ \alpha, \beta \} \subseteq \Phi(W, S)$ with $o(r_{\alpha} r_{\beta}) = \infty$. As a consequence, all classical RGD~systems are linearizable. In addition, many RGD~systems can easily be seen to be linearizable (see below). However, the general situation is very different, as shown by our main result.

\begin{thmintro}\label{Thmintro: Existence of non-linearizable examples}
	There exist uncountably many RGD~systems which are not linearizable.
\end{thmintro}

Following \cite{Bi22}, an RGD~system $\mc{D} = (G, (U_{\alpha})_{\alpha \in \Phi(W, S)})$ of type $(W, S)$ satisfies condition (nc) if for each prenilpotent pair $\{\alpha, \beta\} \subseteq \Phi(W, S)$ with $o(r_{\alpha} r_{\beta}) = \infty$ the root groups $U_{\alpha}$ and $U_{\beta}$ commute. In particular, such RGD~systems are linearizable (see Lemma~\ref{Lemma: nc implies linearizable}). In many examples of RGD~systems it is known that they are linear w.r.t.\ a concrete root basis:
\begin{enumerate}[label=(\alph*)]
	\item Suppose that $(W, S)$ is simply-laced, that is, $m_{st} \in \{2, 3\}$ for all distinct $s, t \in S$. Then $\mc{D}$ satisfies (nc) (see \cite[Theorem~C]{Bi22} and also \cite[Lemma~6]{AC16}).
	
	\item\label{Item: Complete diagram implies nc} Suppose that for all distinct $s, t \in S$ we have $m_{st} \notin \{2, \infty\}$, and that the root groups are not too small (e.g.\ for each root $\alpha \in \Phi(W, S)$ we have $\vert U_{\alpha} \vert \geq 4$). Then $\mc{D}$ satisfies (nc) (see \cite[Theorem~C]{Bi22}).
	
	\item All \emph{exotic} RGD~systems of right-angled type (i.e.\ $m_{st} \in \{2, \infty\}$ for all distinct $s, t \in S$) constructed in \cite{RR06} by Rémy and Ronan satisfy (nc).
	
	\item Let $A$ be a generalized Cartan matrix and suppose that $\mc{D}$ is the RGD~system of a Kac--Moody group of type $A$ over a field. Then $\mc{D}$ `often' satisfies (nc) (see \cite[Remark~3.7(f)]{Ti87}). However, $\mc{D}$ is always linear w.r.t.\ $\mc{B}(A)$ by definition of the commutation relations. Moreover, if $\{\alpha, \beta\} \subseteq \Phi(W, S)$ is a prenilpotent pair with $o(r_\alpha r_\beta) = \infty$, then one can show that $[U_\alpha, U_\beta] \leq U_{\alpha + \beta}$ (see \cite[Proposition~1]{BP95} and \cite[Theorem~2]{Mo87}).
	
	\item Suppose that $(W, S)$ is of type $\tilde{A}_1$. Then one can show that $[\alpha, \beta]_{\lin}^{\mc{B}} = [\alpha, \beta]$ for all prenilpotent pairs $\{\alpha, \beta\} \subseteq \Phi(W, S)$ and any reduced root basis $\mc{B}$ associated with $(W, S)$ (see Lemma~\ref{Lemma: rank 2 intervals coindice}). In particular, the uncountably many RGD~systems of type $\tilde{A}_1$ over $\FF_2$ (i.e.\ each root group has cardinality $2$) constructed by Tits in \cite[Section~5.4 in the academic year $1995/96$]{Ti73-00} are linearizable.
	
	\item Suppose that $(W, S)$ is of affine type. Let $A$ be the corresponding generalized Cartan matrix of untwisted type from \cite[Table Aff~1]{Kac90}. Then one can show that $[\alpha, \beta]_{\lin}^{\mc{B}(A)} = [\alpha, \beta]$ for all prenilpotent pairs $\{ \alpha, \beta \} \subseteq \Phi(W, S)$ (see also \cite[Proposition~6.3~(a)]{Kac90}).
\end{enumerate}

A Coxeter system $(W, S)$ is said to be \emph{of type $(4, 4, 4)$}, if $S$ has cardinality three and $m_{st} = 4$ for all distinct $s, t \in S$. Recently (in his PhD thesis) the author constructed uncountably many examples of RGD~systems of type $(4, 4, 4)$ over $\FF_2$ with non-trivial commutation relations (see \cite{Bischof:Construction_of_commutator_blueprints}, \cite{BiUncountablymanyRGD-systems} and \cite{BiRGD}). In this article we show that these RGD~systems provide uncountably many examples of RGD~systems which are not linearizable. As a consequence, we obtain the following result which immediately implies Theorem~\ref{Thmintro: Existence of non-linearizable examples} (see Theorem~\ref{Theorem: 444 not linear wrt B(W,S)}):

\begin{thmintro}\label{Thmintro: 444}
	There exist uncountably many RGD~systems of type $(4, 4, 4)$ over $\FF_2$ which are not linearizable.
\end{thmintro}

\begin{rkintro}
	Let $\mc{D}$ be an RGD~system of type $(4, 4, 4)$. If all root groups contain at least three elements, then $\mc{D}$ is linearizable. Even more, it is a consequence of \ref{Item: Complete diagram implies nc} that $\mc{D}$ satisfies (nc). In particular, the commutation relations are well-behaved. If all root groups have cardinality two, the situation is completely different. There exists an example (provided by any Kac--Moody group of type $(4, 4, 4)$ over $\FF_2$) which satisfies (nc), hence, which is linearizable. But by Theorem~\ref{Thmintro: 444} there are also uncountably many examples which are not linearizable.
\end{rkintro}

\begin{rkintro}
	In general, the commutation relations of RGD~systems of $2$-spherical type (i.e.\ $m_{st} < \infty$ for all $s, t \in S$) are determined by their `local' commutation relations if the root groups are not too small (see \cite{AM97}). This implies that the commutation relations are well-behaved for almost all $2$-spherical RGD~systems. Hence, the existence of non-linearizable RGD~systems of $2$-spherical type is surprising.
\end{rkintro}

We also construct uncountably many RGD~systems of universal type (i.e.\ $m_{st} = \infty$ for all distinct $s, t \in S$) which are not linearizable. The construction is done by constructing certain commutator blueprints\footnote{Commutator blueprints were introduced in \cite{BiRGD} as purely combinatorial objects in order to construct new RGD~systems.}, and the non-linearizability is based on Rémy's example of $[\alpha, \beta]_{\lin}^{\mc{B}} \subsetneq [\alpha, \beta]$ (see \cite[p.\ 130]{Re03}; this is essentially Lemma~\ref{Lemma: Extension Rémy} in the case $n=0$). The following result is Corollary~\ref{Corollary: uncountably many non-linearizable}.

\begin{thmintro}\label{Thmintro: not inearizable}
	There exist uncountably many RGD~systems of universal type which are not linearizable.
\end{thmintro}

\subsection*{Overview}

In Section~2 we recall the basics of Coxeter systems, root bases, and RGD~systems. In addition, we prove some useful results about them. Section~3 is devoted to the proof of Theorem~\ref{Thmintro: not inearizable}. We will construct uncountably many RGD~systems and show that they are not linearizable. In Section~4 we prove Theorem~\ref{Thmintro: 444} by recalling the existence of the RGD~systems of type $(4, 4, 4)$ over $\FF_2$ recently constructed by the author and prove that they are not linearizable.

\section{Preliminaries}\label{Section: Preliminaries}

\subsection*{Coxeter systems}

Let $(W, S)$ be a Coxeter system and denote by $\ell\co W \to \NN$ its length function. For all $s, t \in S$ we denote by $m_{st}$ the order of the product $st$ in $W$. The \emph{rank} of $(W, S)$ is the cardinality of $S$. The Coxeter system $(W, S)$ is said to be of \emph{universal type} if $m_{st} = \infty$ for all distinct $s, t \in S$. Moreover, $(W, S)$ is of \emph{type $(4, 4, 4)$} if $(W, S)$ has rank three and $m_{st} = 4$ for all distinct $s, t \in S$. It is well known that for each $J \subseteq S$ the pair $(\gen{J}, J)$ is a Coxeter system (see \cite[Chapter~IV, Section~1, Theorem~2]{Bo02}).

\subsection*{The chamber system $\mathbf{\Sigma(W, S)}$}

Let $(W, S)$ be a Coxeter system. Defining $w \sim_s w'$ if and only if $w\inv w' \in \gen{s}$, we obtain a chamber system with chamber set $W$ and equivalence relations $\sim_s$ for $s\in S$, which we denote by $\Sigma(W, S)$. We call two chambers $w, w'$ \emph{$s$-adjacent} if $w \sim_s w'$, and \emph{adjacent} if they are $s$-adjacent for some $s\in S$. A \emph{gallery of length $n$} from $w_0$ to $w_n$ is a sequence $(w_0, \ldots, w_n)$ of chambers where $w_i$ and $w_{i+1}$ are adjacent for all $0 \leq i < n$. A gallery $(w_0, \ldots, w_n)$ is called \emph{minimal} if there exists no gallery from $w_0$ to $w_n$ of length $k<n$. Let $G = (w_0, \ldots, w_n)$ be a minimal gallery. Then, $s_i := w_{i-1}\inv w_i \in S$ for all $1 \leq i \leq n$, and we call $(s_1, \ldots, s_n)$ the \emph{type} of $G$.

\subsection*{Roots of Coxeter systems viewed as half-spaces}

Let $(W, S)$ be a Coxeter system. For $s\in S$ we define the \emph{simple root corresponding to $s\in S$} by $\alpha_s = \{ w\in W \mid \ell(sw) > \ell(w) \}$. A \emph{root} of $(W, S)$ is a translate $w\alpha_s$ for some $w\in W$ and $s\in S$. We denote the set of all roots by $\Phi(W, S)$. A root is called \emph{positive}, if it contains $1_W$; otherwise it is called \emph{negative}. We denote the set of all positive/negative roots by $\Phi(W, S)_{\pm}$. A pair of roots $\{ \alpha, \beta \} \subseteq \Phi(W, S)$ is said to be \emph{prenilpotent} if there exist $w, w' \in W$ such that $\{w\alpha, w\beta\} \subseteq \Phi(W, S)_+$ and $\{ w'\alpha, w'\beta \} \subseteq \Phi(W, S)_-$. For a prenilpotent pair $\{ \alpha, \beta \} \subseteq \Phi(W, S)$ we define the \emph{(geometric) interval}
\[ [\alpha, \beta] := \bigcap_{\substack{w\in W \\ \epsilon \in \{+, -\}}} \{ \gamma \in \Phi(W, S) \mid \{ w\alpha, w\beta \} \subseteq \Phi(W, S)_{\epsilon} \Rightarrow w\gamma \in \Phi(W, S)_{\epsilon}  \} \]
as well as $(\alpha, \beta) := [\alpha, \beta] \setminus \{ \alpha, \beta \}$. For each root $\alpha \in \Phi(W, S)$, we denote its opposite root by $-\alpha := W \setminus \alpha \in \Phi(W, S)$ and we denote the unique reflection that interchanges these two roots by $r_{\alpha}$. For a minimal gallery $G = (w_0, \ldots, w_k)$, the \emph{sequence of roots crossed by $G$} is the unique sequence of roots $(\alpha_1, \ldots, \alpha_k)$ such that $w_{i-1} \in \alpha_i$ and $w_i \notin \alpha_i$ for all $1 \leq i \leq k$. In this case we say that $G$ is a minimal gallery \emph{between} $\alpha_1$ and $\alpha_k$. For $w\in W$ we define $\Phi(W, S)_w := \{ \alpha \in \Phi_+ \mid w \notin \alpha \}$. Note that this set is finite. Indeed, let $G$ be a minimal gallery from $1_W$ to $w$ and let $(\alpha_1, \ldots, \alpha_k)$ be the sequence of roots crossed by $G$. Then $\Phi(W, S)_w = \{ \alpha_1, \ldots, \alpha_k\}$.

\begin{remark}\label{Remark: intervals}
	Let $\{ \alpha, \beta \} \subseteq \Phi(W, S)$ be a prenilpotent pair. It follows right from the definition of $[\alpha, \beta]$ that $[\alpha, \beta] = \{ \gamma \in \Phi(W, S) \mid \alpha \cap \beta \subseteq \gamma \text{ and } (-\alpha) \cap (-\beta) \subseteq (-\gamma) \}$.
\end{remark}

\begin{lemma}\label{Lemma: rank 2 intervals}
	Suppose that $(W, S)$ is of rank $2$. Let $G$ be a minimal gallery from $1_W$ to some $w\in W$ and let $(\alpha_1, \ldots, \alpha_k)$ be the sequence of roots crossed by $G$. Then $[\alpha_i, \alpha_j] = \{ \alpha_i, \ldots, \alpha_j \}$ for all $1 \leq i \leq j \leq k$.
\end{lemma}
\begin{proof}
	Let $\alpha \in [\alpha_i, \alpha_j]$. Then $c_{i-1} \in \alpha_i \cap \alpha_j \subseteq \alpha$ and $c_j \in (-\alpha_i) \cap (-\alpha_j) \subseteq (-\alpha)$ by \cite[Lemma~3.69]{AB08}. Moreover, \cite[Lemma~3.69]{AB08} implies $\alpha \in \{\alpha_i, \ldots, \alpha_j\}$. For the other inclusion we let $S = \{s, t\}$ and we distinguish the following two cases:
	\begin{itemize}
		\item If $m_{st} = \infty$, then $\alpha_i \subseteq \alpha_j$ (see \cite[Lemma~2.8~(a)]{Bischof:Construction_of_commutator_blueprints}) and $(-\alpha_j) \subseteq (-\alpha_i)$ for all $1 \leq i \leq j \leq k$. For $i \leq n \leq j$ we have
		\[ \alpha_i \cap \alpha_j = \alpha_i \subseteq \alpha_n \qquad\text{and}\qquad (-\alpha_i) \cap (-\alpha_j) = (-\alpha_j) \subseteq (-\alpha_n). \]
		Now Remark~\ref{Remark: intervals} shows that $\alpha_k \in [\alpha_i, \alpha_j]$.
		
		\item Suppose now that $m := m_{st} < \infty$. Let $r_S \in W$ be the unique element of maximal length (see \cite[Corollary~2.19]{AB08}). We can extend $G$ to a minimal gallery $H = (h_1, \ldots, h_m)$ from $1_W$ to $r_S$. We let $(\beta_1, \ldots, \beta_m)$ be the sequence of roots crossed by $H$. Let $H' = (h_1', \ldots, h_m')$ be the unique minimal gallery from $1_W$ to $r_S$ which is different from $H$, i.e.\ if $H$ has type $(s, t, \ldots)$, then $H'$ has type $(t, s, \ldots)$. Let $1 \leq i \leq n \leq j \leq m$ and suppose that $\beta_n \notin [\beta_i, \beta_j]$. Then there exist $1 \leq n' \leq m$ such that $h_{n'}' \in \alpha_i \cap \alpha_j \cap (-\alpha_n)$. Now \cite[Proposition~29.18]{We09} yields that $n-1 + n' > m_{st}$ as well as $j-1 + n' \leq m_{st}$. This is a contradiction to the fact that $n \leq j$. \qedhere
	\end{itemize}
\end{proof}

\subsection*{Root bases}

This section is based on \cite{CR09b}.
A \emph{root basis} is a tuple $\mc{B} = (V, \Pi, \Pi^\vee = \{ \alpha^\vee \mid \alpha \in \Pi \})$ consisting of a real vector space $V$, a non-empty set $\Pi \subseteq V$, and a set of linear forms $\Pi^\vee \subseteq V^*$ satisfying the following axioms, where $\gen{\alpha, \beta^\vee} := \beta^\vee(\alpha)$:
\begin{enumerate}[label=(RB\arabic*)]
	\item For each $\alpha \in \Pi$ we have $\gen{\alpha, \alpha^\vee} = 2$.
	
	\item For all $\alpha, \beta \in \Pi$ with $\alpha \neq \beta$ one of the following hold:
	\begin{itemize}
		\item $\gen{\alpha, \beta^\vee} = 0 = \gen{\beta, \alpha^\vee}$;
		
		\item $\gen{\alpha, \beta^\vee}, \gen{\beta, \alpha^\vee} <0$ and $\gen{\alpha, \beta^\vee} \gen{\beta, \alpha^\vee} \in \left\{ 4 \cos^2 \left( \frac{\pi}{k} \right) \mid k \in \NN \right\} \cup \RR_{\geq 4}$.
	\end{itemize}
	
	\item\label{Item: RB3} There exists $f \in V^*$ such that $\gen{\alpha, f} >0$ for all $\alpha \in \Pi$.
\end{enumerate}

Note that \ref{Item: RB3} is always satisfied if $\Pi$ is linearly independent. Let $\mc{B} = (V, \Pi, \Pi^\vee)$ be a root basis. The \emph{Cartan matrix of $\mc{B}$} is the matrix $A(\mc{B}) := (A_{_{\alpha, \beta}})_{\alpha, \beta \in \Pi}$ defined by $A_{\alpha, \beta} := \gen{\alpha, \beta^\vee}$. For each $\alpha \in \Pi$ we define the reflection $r_\alpha\co V \to V, \, v \mapsto v - \gen{v, \alpha^\vee} \alpha$. Moreover, we define $S = S(\mc{B}) = \{ r_\alpha \mid \alpha \in \Pi \}$ and $W = W(\mc{B}) = \gen{S(\mc{B})} \leq \GL(V)$. It is a fact that the pair $(W, S)$ is a Coxeter system (see \cite[Theorem~1.3(i)]{CR09b}). Furthermore, for all distinct $\alpha, \beta \in \Pi$, the order $o(r_\alpha r_\beta)$ of $r_\alpha r_\beta$ is given by
\[ o(r_\alpha r_\beta) = \begin{cases}
	k, & A_{\alpha, \beta} A_{\beta, \alpha} = 4\cos^2\left( \frac{\pi}{k} \right), \\
	\infty, & A_{\alpha, \beta} A_{\beta, \alpha} \geq 4.
\end{cases} \]

The \emph{root system of $\mc{B}$} is the set $\Phi(\mc{B}) = \{ w(\alpha) \mid w\in W, \alpha \in \Pi \}$. Moreover, we define
\[ \Phi(\mc{B})_+ := \Phi(\mc{B}) \cap \left( \sum\nolimits_{\alpha \in \Pi} \RR_{\geq 0} \alpha \right) \qquad\text{and}\qquad \Phi(\mc{B})_- := \Phi(\mc{B}) \cap \left( \sum\nolimits_{\alpha \in \Pi} \RR_{\leq 0} \alpha \right). \]
It is a fact that $\Phi(\mc{B}) = \Phi(\mc{B})_+ \cup \Phi(\mc{B})_-$ (see \cite[Theorem~1.3(ii)]{CR09b}). The root basis $\mc{B}$ is called \emph{reduced} if $\Phi(\mc{B}) \cap \RR \alpha = \{ \alpha, -\alpha \}$ for all $\alpha \in \Pi$. By \cite[Lemma~$1.7$]{CR09b} the root basis $\mc{B}$ is reduced if and only if for all distinct $\alpha, \beta \in \Pi$ such that the order $o(r_{\alpha} r_{\beta})$ is finite and odd, one has $A_{\alpha, \beta} = A_{\beta, \alpha}$.

\begin{example}
	\begin{enumerate}[label=(\alph*)]
		\item Let $(W, S)$ be a Coxeter system. Let $V$ be a real vector space with basis $\Pi := \{ \alpha_s \mid s \in S \}$. Let $(\cdot, \cdot)\co V \times V \to \RR$ denote the symmetric bilinear form on $V$ defined by
		\[ (\alpha_s, \alpha_t) := \begin{cases}
			-\cos \left( \frac{\pi}{m_{st}} \right), & m_{st} < \infty, \\
			-1, & m_{st} = \infty.
		\end{cases} \]
		For each $s\in S$ we define a linear form
		\[ \alpha_s^{\vee}\co V \to \RR, \, v \mapsto 2 (v, \alpha_s). \]
		Set $\Pi^\vee := \{ \alpha_s^\vee \mid s\in S \}$. Then $\mc{B}(W, S) := (V, \Pi, \Pi^\vee)$ is a root basis.
		
		\item Let $I$ be a finite set and let $A = (a_{ij})_{i, j \in I}$ be a generalized Cartan matrix. Let $V$ be a real vector space with basis $\Pi = \{ \alpha_i \mid i \in I \}$. For $i \in I$ we define the linear form 
		\[ \alpha_i^\vee\co V \to \RR, \, \alpha_j \mapsto \gen{\alpha_j, \alpha_i^\vee} := a_{ij} \qquad\text{for all } j\in I. \]
		Set $\Pi^\vee = \{ \alpha_i^\vee \mid i \in I \}$. Then $\mc{B}(A) := (V, \Pi, \Pi^\vee)$ is a root basis.
	\end{enumerate}
\end{example}

\subsection*{Root bases for Coxeter systems}

Let $(W, S)$ be a Coxeter system and let $\mc{B} = (V, \Pi, \Pi^\vee)$ be a root basis. Recall that $(W(\mc{B}), S(\mc{B}))$ is a Coxeter system. We say that the root basis $\mc{B}$ is \emph{associated with} $(W, S)$ if there exists a bijection $S \to S(\mc{B})$ which extends to an isomorphism $W \to W(\mc{B})$ (equivalently, there exists an isomorphism $W \to W(\mc{B})$ which maps $S$ to $S(\mc{B})$). In this case we relax notation and identify $W$ with $W(\mc{B})$

\begin{example}
	The root basis $\mc{B}(W, S)$ is associated with $(W, S)$.
\end{example}

The following lemma is essentially obtained from \cite[B.4]{Ma18} and \cite[$\S$5.1]{Re03}.

\begin{lemma}\label{Lemma: W-equivariant bijection}
	Let $\mc{B}$ be a root basis associated with $(W, S)$. For $\alpha \in \Phi(\mc{B})$ we define $\alpha_{\halfspace} := \{ w\in W \mid w\inv(\alpha) \in \Phi(\mc{B})_+ \}\footnote{$\halfspace$ stands for half-space}$.
	\begin{enumerate}[label=(\alph*)]
		\item For $\alpha \in \Phi(\mc{B})$ we have $\alpha_{\halfspace} \in \Phi(W, S)$.
		
		\item The mapping $\phi\co \Phi(\mc{B}) \to \Phi(W, S), \, \alpha \mapsto \alpha_{\halfspace}$ is well-defined, $W$-equivariant and surjective. In addition, if $\mc{B}$ is reduced, then $\phi$ is bijective.
	\end{enumerate}
\end{lemma}
\begin{proof}
	Let $\mc{B} = (V, \Pi, \Pi^\vee)$. Before we start the prove we observe the following for $w\in W$ and $\alpha \in \Phi(\mc{B})$:
	\allowdisplaybreaks
	\begin{align*}
		w(\alpha)_{\halfspace} &= \{ v \in W \mid v\inv(w(\alpha)) \in \Phi(\mc{B})_+ \} \\
		&= \{ v \in W \mid (w\inv v)\inv (\alpha) \in \Phi(\mc{B})_+ \} \\
		&= w \{ v\in W \mid v\inv(\alpha) \in \Phi(\mc{B})_+ \} \\
		&= w \alpha_{\halfspace}.
	\end{align*}
	
	(a) Note that for $\alpha \in \Pi$ and $w\in W$ the following hold (see \cite[Lemma~1.5(i)]{CR09b}):
	\allowdisplaybreaks
	\begin{align*}
		\ell(r_\alpha w) > \ell(w) \Longleftrightarrow w\inv(\alpha) \in \Phi(\mc{B})_+ \qquad\text{and}\qquad \ell(r_\alpha w) < \ell(w) \Longleftrightarrow w\inv(\alpha) \in \Phi(\mc{B})_-.
	\end{align*}
	We first proof the claim for $\alpha \in \Pi$. Let $s = r_\alpha \in S$. Then the following hold:
	\[ \alpha_{\halfspace} = \{ w\in W \mid w\inv(\alpha) \in \Phi(\mc{B})_+ \} = \{ w\in W \mid \ell(r_\alpha w) > \ell(w) \} = \alpha_s. \]
	Now let $\alpha \in \Phi(\mc{B})$. Then there exist $v\in W$ and $\beta \in \Pi$ with $\alpha = v(\beta)$. Let $s = r_{\beta} \in S$. Then the following hold:
	\allowdisplaybreaks
	\begin{align*}
		\alpha_{\halfspace} &= v(\beta)_{\halfspace} = v \beta_{\halfspace} = v \alpha_s.
	\end{align*}
	(b) By (a), the mapping $\phi$ is well-defined. Let $\alpha \in \Phi(\mc{B})$ and $w\in W$. Then $\phi$ is $W$-equivariant, as shown by the following:
	\allowdisplaybreaks
	\begin{align*}
		\phi(w(\alpha)) = w(\alpha)_{\halfspace} &= w \alpha_{\halfspace} = w \phi(\alpha).
	\end{align*}
	For $w\in W$ and $s = r_\alpha \in S$ for some $\alpha \in \Pi$, we have $\phi(w(\alpha)) = w \alpha_{\halfspace} = w\alpha_s$, hence $\phi$ is surjective. Now suppose that $\Phi(\mc{B})$ is reduced, and let $\alpha, \beta \in \Phi(\mc{B})$ be distinct. We claim that there exists $w\in W$ with
	\[ \{ w(\alpha), w(\beta) \} \cap \Phi(\mc{B})_{\pm} \neq \emptyset. \]
	To see this, we let $v\in W$ with $\delta := v(\alpha) \in \Pi \subseteq \Phi(\mc{B})_+$. If $v(\beta) \in \Phi(\mc{B})_-$, the claim follows with $w = v$. Thus we can assume $v(\beta) \in \Phi(\mc{B})_+$. Hence, there exist $p_{\gamma} \in \RR_{\geq 0}$ with $v(\beta) = \sum\nolimits_{\gamma \in \Pi} p_\gamma \gamma$. As $\alpha$ and $\beta$ are distinct, and as $\Phi(\mc{B})$ is reduced, there exists $\gamma \in \Pi \setminus \{ \delta \}$ with $p_{\gamma} >0$. But then the coefficient of $\gamma$ of $r_{\delta} v(\beta)$ is $p_{\gamma} >0$. This implies $r_{\delta} v(\beta) \in \Phi(\mc{B})_+$. As $r_{\delta} v(\alpha) = -v(\alpha) \in \Phi(\mc{B})_-$, the claim follows with $w = r_{\delta} v$.
	
	This implies that $w\inv \in \left( \alpha_{\halfspace} \cup \beta_{\halfspace} \right) \setminus \left( \alpha_{\halfspace} \cap \beta_{\halfspace} \right)$. In particular, $\phi(\alpha) \neq \phi(\beta)$ and $\phi$ is injective.
\end{proof}

\subsection*{RGD~systems}

Let $(W, S)$ be a Coxeter system. An \emph{RGD~system of type $(W, S)$} is a pair $\mc{D} = (G, (U_{\alpha})_{\alpha \in \Phi(W, S)})$ consisting of a group $G$ and a family of subgroups $U_{\alpha}$ called \emph{root groups} indexed by the set of roots $\Phi(W, S)$ satisfying some axioms. For the precise definition we refer the reader to \cite[Sections~7.8 and 8.6]{AB08}. The most important axiom makes an assumption about commutation relations between root groups corresponding to prenilpotent pairs $\{ \alpha, \beta \} \subseteq \Phi(W, S)$ with $\alpha \neq \beta$:
\[ [U_{\alpha}, U_{\beta}] \leq \gen{ U_{\gamma} \mid \gamma \in (\alpha, \beta) }. \]
For any $w\in W$ we define the group $U_w$ as follows:
\[ U_w := \gen{ U_\gamma \mid \gamma \in \Phi(W, S)_w }. \]
Let $H$ be a minimal gallery from $1_W$ to $w$ and let $(\alpha_1, \ldots, \alpha_k)$ be the sequence of roots crossed by $H$. Then the multiplication map
\begin{equation}\label{Equation: multiplication map bijection}
	U_{\alpha_1} \times \cdots \times U_{\alpha_k} \to U_w, \, (g_1, \ldots, g_k) \mapsto g_1 \cdots g_k
\end{equation}
is a bijection (see \cite[Corollary~8.34]{AB08}).

\subsection*{Linear RGD~systems}

Let $(W, S)$ be a Coxeter system, let $\mc{D} = (G, (U_{\alpha})_{\alpha \in \Phi(W, S)})$ be an RGD~system of type $(W, S)$, and let $\mc{B}$ be a reduced root basis associated with $(W, S)$. By Lemma~\ref{Lemma: W-equivariant bijection} there exists a $W$-equivariant bijection $\phi\co \Phi(W, S) \to \Phi(\mc{B})$. For a prenilpotent pair $\{\alpha, \beta\} \subseteq \Phi(W, S)$ we define the \emph{(algebraic) interval} (w.r.t.\ $\mc{B}$) by
\[ [\alpha, \beta]_{\lin}^{\mc{B}} := \phi\inv \left( \RR_{\geq 0} \phi(\alpha) + \RR_{\geq 0} \phi(\beta) \right) \cap \Phi(W, S) \]
as well as $(\alpha, \beta)_{\lin}^{\mc{B}} := [\alpha, \beta]_{\lin}^{\mc{B}} \setminus \{ \alpha, \beta \}$. It is a fact that $[\alpha, \beta]_{\lin}^{\mc{B}} \subseteq [\alpha, \beta]$. We say that $\mc{D}$ is \emph{linear w.r.t.\ $\mc{B}$} if for each prenilpotent pair $\{\alpha, \beta\} \subseteq \Phi(W, S)$ with $\alpha \neq \beta$ the following holds:
\[ [U_{\alpha}, U_{\beta}] \leq \gen{ U_{\gamma} \mid \gamma \in \left( \alpha, \beta \right)_{\lin}^{\mc{B}} }. \]
We say that $\mc{D}$ is \emph{linearizable} if there exists a reduced root basis $\mc{B}$ associated with $(W, S)$ such that $\mc{D}$ is linear w.r.t.\ $\mc{B}$.

\begin{lemma}\label{Lemma: non-linear RGD system}
	Let $(W, S)$ be a Coxeter system and let $\mc{D} = (G, (U_{\alpha})_{\alpha \in \Phi(W, S)})$ be an RGD~system of type $(W, S)$. Let $\{\alpha, \beta\} \subseteq \Phi(W, S)$ be a prenilpotent pair with $\alpha \neq \beta$, and let $u_\alpha \in U_\alpha, u_\beta \in U_\beta$. Suppose $(\alpha, \beta) = \{ \gamma_1, \ldots, \gamma_k \}$, and $u_i \in U_{\gamma_i}$ for $1 \leq i \leq k$ with
	\[ [u_\alpha, u_{\beta}] = u_1 \cdots u_k. \]
	Let $\mc{B}$ be a reduced root basis associated with $(W, S)$. If there exists $1 \leq i \leq k$ with $u_i \neq 1$ and $\gamma_i \notin (\alpha, \beta)_{\lin}^{\mc{B}}$, then $\mc{D}$ is not linear w.r.t.\ $\mc{B}$.
\end{lemma}
\begin{proof}
	This follows immediately from (\ref{Equation: multiplication map bijection}).
\end{proof}

\begin{lemma}\label{Lemma: rank 2 intervals coindice}
	Let $\{ \alpha, \beta \} \subseteq \Phi(W, S)$ be a prenilpotent pair. Suppose that there exist $w\in W$ and $s, t\in S$ distinct such that $\{ w\alpha, w\beta \} \subseteq \Phi(\gen{s, t}, \{s, t\})$. Then $[\alpha, \beta]_{\lin}^{\mc{B}} = [\alpha, \beta]$ for any reduced root basis $\mc{B}$ associated with $(W, S)$.
\end{lemma}
\begin{proof}
	Let $\mc{B}$ be a reduced root basis associated with $(W, S)$, and let $\phi\co \Phi(W, S) \to \Phi(\mc{B})$ be the $W$-equivariant bijection from Lemma~\ref{Lemma: W-equivariant bijection}. Note first that
	\begin{equation}\label{Equation: intervals W-equivariant}
		v [\alpha, \beta]_{\lin}^{\mc{B}} = [v\alpha, v\beta]_{\lin}^{\mc{B}}
	\end{equation}
	for any $v\in W$, as $\phi(\Phi(W, S)_+) = \Phi(\mc{B})_+$. Moreover, we have $v[\alpha, \beta] = [v\alpha, v\beta]$. Hence it suffices to show the claim for $\alpha = \alpha_s$ and $\beta = s_1 \cdots s_{k-1} \alpha_{s_k}$, where $(s_1, \ldots, s_k) = (s, t, \ldots)$ and $\ell(s_1 \cdots s_k) = k$. Note that in this case $\beta \in \Phi(W, S)_+$ and there exists a minimal gallery between $\alpha$ and $\beta$ of type $(s_1, \ldots, s_k)$. We show the claim by induction on $k$. If $k = 1$, then $\alpha = \beta$ and there is nothing to show. Thus we can assume $k \geq 2$. Let $\gamma = s_1 \alpha_{s_2} = s \alpha_t$. Note that $[\alpha, \beta] = \{ \alpha \} \cup [\gamma, \beta]$ by Lemma~\ref{Lemma: rank 2 intervals}. Using induction, we have $[\gamma, \beta]_{\lin}^{\mc{B}} = [\gamma, \beta]$. Hence it suffices to show $\gamma \in [\alpha, \beta]_{\lin}^{\mc{B}}$. Note that the positive root $\beta$ is different from $\alpha_{s_1} = \alpha_s$ (see \cite[Lemma~3.69]{AB08}). Thus there exist $b_s \in \RR_{\geq 0}$ and $b_t \in \RR_{>0}$ with 
	\[ \phi(\beta) = b_s \phi(\alpha_s) + b_t \phi(\alpha_t). \]
	This implies that
	\[ \phi(\gamma) = \phi(s_1(\alpha_{s_2})) = s \phi(\alpha_t) = \phi(\alpha_t) - a_{st} \phi(\alpha_s) = \left( -a_{st} - \frac{b_s}{b_t} \right) \phi(\alpha_s) + \frac{1}{b_t} \phi(\beta). \]
	Now suppose that $-a_{st} - \frac{b_s}{b_t} <0$. Then we would have $\gamma \in [-\alpha, \beta]_{\lin}^{\mc{B}} \subseteq [-\alpha, \beta]$ and, hence, Remark~\ref{Remark: intervals} would yield that
	\[ \beta = ((-\alpha) \cap \beta) \cup (\alpha \cap \beta) \subseteq \gamma. \]
	This is a contradiction.
\end{proof}

\begin{lemma}\label{Lemma: nc implies linearizable}
	Any RGD~system $\mc{D}$ satisfying condition (nc) is linearizable.
\end{lemma}
\begin{proof}
	Let $\mc{D} = (G, (U_{\alpha})_{\alpha \in \Phi(W, S)})$ be an RGD~system of type $(W, S)$, let $\mc{B}$ be a reduced root basis associated with $(W, S)$, and let $\{ \alpha, \beta \} \subseteq \Phi(W, S)$ be a prenilpotent pair with $\alpha \neq \beta$. If $o(r_{\alpha} r_{\beta}) < \infty$, then $[\alpha, \beta] = [\alpha, \beta]_{\lin}^{\mc{B}}$ by Lemma~\ref{Lemma: rank 2 intervals coindice} (see also \cite[Lemma~2.8]{Bischof:Construction_of_commutator_blueprints}). In particular, we have $[U_{\alpha}, U_{\beta}] \leq \gen{ U_{\gamma} \mid \gamma \in [\alpha, \beta]_{\lin}^{\mc{B}} }$. If $o(r_{\alpha} r_{\beta}) = \infty$, then $\{ \alpha, \beta \}$ is \emph{nested}, that is, $\alpha \subseteq \beta$ or $\beta \subseteq \alpha$ (see \cite[Proposition~3.165 \& Lemma~8.42(3)]{AB08}). Now condition (nc) yields that $[U_{\alpha}, U_{\beta}] = 1$. We conclude that $\mc{D}$ is linearizable.
\end{proof}

\begin{remark}\label{Remark: independent intervals}
	We believe that the interval $[\alpha, \beta]_{\mathrm{lin}}^{\mathcal{B}}$ is independent of the root basis $\mathcal{B}$. In some special cases (e.g.\ $(W, S)$ is of simply-laced type) the action of $S \subseteq W$ on the corresponding vector space is independent of the root basis $\mathcal{B}$. In particular, the interval is independent of $\mathcal{B}$.
	
	Let us discuss the general case. If $o(r_{\alpha} r_{\beta}) < \infty$, then we have $[\alpha, \beta]_{\mathrm{lin}}^{\mathcal{B}} = [\alpha, \beta]$ by Lemma~\ref{Lemma: rank 2 intervals coindice}. Hence, the interval is independent of $\mathcal{B}$. If $o(r_{\alpha} r_{\beta}) = \infty$, it suffices by (\ref{Equation: intervals W-equivariant}) to consider the case $\alpha = \alpha_s$ for some $s\in S$. A first natural step would be to understand these intervals between roots of Coxeter systems of universal type.
\end{remark}

\section{Non-linearizable RGD~systems of universal type}\label{Section: Non-linearizable universal type}

In this section we let $(W, S)$ be a Coxeter system of universal type of rank at least three, and we let $r, s, t \in S$ be pairwise distinct. For $n\in \NN$ we define $\alpha_n := (st)^n s \alpha_t$. The following lemma is a generalization of Rémy's example \cite[p.\ 130]{Re03}.

\begin{lemma}\label{Lemma: Extension Rémy}
	Let $\mc{B}$ be a reduced root basis associated with $(W, S)$. Then we have $\alpha_s \notin [-\alpha_r, \alpha_n]_{\lin}^{\mc{B}}$.
\end{lemma}
\begin{proof}
	Let $\phi\co \Phi(W, S) \to \Phi(\mc{B})$ be the $W$-equivariant bijection from Lemma~\ref{Lemma: W-equivariant bijection}.
	Note that for $u \in \{s, t\}$ and all $\lambda, \mu \in \RR$ there exist $\lambda', \mu' \in \RR$ with
	\[ u \left( \lambda \phi(\alpha_s) + \mu \phi(\alpha_t) \right) = \lambda' \phi(\alpha_s) + \mu' \phi(\alpha_t). \]
	This implies $\phi(\alpha_n) = \lambda_n \phi(\alpha_s) + \mu_n \phi(\alpha_t)$ for some $\lambda_n, \mu_n \in \RR$. In particular, as $\alpha_n \neq \pm \alpha_s$, we deduce $\mu_n \neq 0$, and $\phi(\alpha_s)$ is not a linear combination of $\phi(-\alpha_r)$ and $\phi(\alpha_n)$.
\end{proof}

\begin{theorem}\label{Theorem: RGD systems of universal type}
	For each $K \subseteq \NN$ there exists an RGD~system $\mc{D}(K) = (G, (U_{\alpha})_{\alpha \in \Phi(W, S)})$ of type $(W, S)$ over $\FF_2$ with the following properties:
	\begin{itemize}
		\item $\forall n \in K: [U_{-\alpha_r}, U_{\alpha_n}] = U_{\alpha_s}$;
		
		\item $\forall n \in \NN \setminus K: [U_{-\alpha_r}, U_{\alpha_n}] = \triv$.
	\end{itemize}
\end{theorem}
\begin{proof}
	Before we start the proof, we introduce some notation. A minimal gallery $H$ is said to be \emph{of type $(n)$} for $n \in \NN$ if the type of $H$ is given by $(r, s, t, \ldots, s, t)$, where $s$ (and also $t$) appears $n+1$ times in the type of $H$. For instance, a minimal gallery of type $(2)$ is a minimal gallery of type $(r, s, t, s, t, s, t)$. Note that if $H$ is a minimal gallery of type $(n)$ from $1_W$ to $r(st)^{n+1}$, then $H$ is a minimal gallery between $\alpha_r$ and $r\alpha_n$.
	
	We now start the proof. We use the notation from \cite{Bischof:Construction_of_commutator_blueprints}. We first construct a commutator blueprint $\mc{M} = (M_{\alpha, \beta}^H)_{(H, \alpha, \beta) \in \mc{I}}$ as follows: Let $(H, \alpha, \beta) \in \mc{I}$ and $n \in K$. Suppose that there exists a minimal gallery $\Gamma$ of type $(n)$ between $\alpha$ and $\beta$, and let $(\beta_1 = \alpha, \ldots, \beta_n = \beta)$ be the corresponding sequence of roots crossed by $\Gamma$. Then we define
	\[ M_{\alpha, \beta}^H = M_{\beta_1, \beta_n}^H := \{ \beta_2 \}. \]
	If no such gallery exists, we define $M_{\alpha, \beta}^H := \emptyset$. This is a pre-commutator blueprint (see \cite[Remark~4.2]{Bischof:Construction_of_commutator_blueprints}). It is not hard to see that $\mc{M}$ satisfies the conditions of \cite[Theorem~3.8]{Bischof:Construction_of_commutator_blueprints}, hence it is a commutator blueprint. It is in addition Weyl-invariant, as $M_{\alpha, \beta}^H$ only depends on a minimal gallery between $\alpha$ and $\beta$ of a given type. By \cite[Remark~2]{BiRGD} it is also faithful and, hence, integrable by \cite[Theorem~A]{BiRGD}. This means that there exists a corresponding RGD~system $\mc{D}(K)$ with the desired commutation relations.
\end{proof}

\begin{corollary}\label{Corollary: uncountably many non-linearizable}
	There exist uncountably many RGD~systems of universal type which are not linearizable.
\end{corollary}
\begin{proof}
	Let $K \subseteq \NN$ and let $\mc{D}(K)$ be the RGD~system of type $(W, S)$ from Theorem~\ref{Theorem: RGD systems of universal type}. Let $\mc{B}$ be a reduced root basis associated with $(W, S)$. Then Lemma~\ref{Lemma: Extension Rémy} together with Lemma~\ref{Lemma: non-linear RGD system} imply that $\mc{D}(K)$ is not linear w.r.t.\ $\mc{B}$. We conclude that there are uncountably many RGD~systems of universal type which are not linearizable.
\end{proof}

\section{Non-linearizable RGD~systems of type $(4, 4, 4)$}\label{Section: Non-linearizable (4,4,4)}

In this section we assume that $(W, S)$ is a Coxeter system of type $(4, 4, 4)$, and that $S = \{r, s, t\}$. Moreover, we define
\[ \gamma := rststrs\alpha_t \qquad\text{as well as}\qquad \gamma' := rststrt\alpha_s. \]
For $n\in \NN$ we define $\alpha_n := \alpha_r$ and $\beta_n := (rstst)^n \alpha_r$.

\begin{lemma}\label{Lemma: gamma/gamma' not in alphan,betan lin}
	For any reduced root basis $\mc{B}$ associated with $(W, S)$, we have $\{ \gamma, \gamma' \} \not\subseteq [\alpha_n, \beta_n]_{\lin}^{\mc{B}}$.
\end{lemma}
\begin{proof}
	Let $\phi\co \Phi(W, S) \to \Phi(\mc{B})$ be the $W$-equivariant bijection from Lemma~\ref{Lemma: W-equivariant bijection}. Note that there exist $a_r, a_s, a_r \in \RR_{\geq 0}$ with $\phi(\beta_n) = a_r \phi(\alpha_r) + a_s \phi(\alpha_s) + a_t \phi(\alpha_t)$. Assume by contrary that $\{ \gamma, \gamma' \} \subseteq [\alpha_n, \beta_n]_{\lin}^{\mc{B}}$. Then there exist $a_{\gamma}, b_{\gamma}, a_{\gamma'}, b_{\gamma'} \in \RR_{\geq 0}$ such that
	\allowdisplaybreaks
	\begin{align*}
		\phi(\gamma) &= a_{\gamma} \phi(\alpha_r) + b_{\gamma} \phi(\beta_n) = (a_{\gamma} + b_{\gamma} a_r) \phi(\alpha_r) + b_{\gamma} a_s \phi(\alpha_s) + b_{\gamma} a_t \phi(\alpha_t), \\
		\phi(\gamma') &= a_{\gamma'} \phi(\alpha_r) + b_{\gamma'} \phi(\beta_n) = (a_{\gamma'} + b_{\gamma'} a_r) \phi(\alpha_r) + b_{\gamma'} a_s \phi(\alpha_s) + b_{\gamma'} a_t \phi(\alpha_t).
	\end{align*}
	As $\gamma \neq \alpha_r \neq \gamma'$ and as $\mc{B}$ is reduced, we obtain $b_{\gamma}, b_{\gamma'} >0$. For $\delta \in \{ \gamma, \gamma' \}$ we define $c_{\delta} := a_{\delta} + b_{\delta} a_r$. We now distinguish the following two cases:
	\begin{itemize}[label=$\bullet$]
		\item $c_{\gamma} b_{\gamma'} = c_{\gamma'} b_{\gamma}$: In this case we have
		\[ \frac{b_{\gamma'}}{b_{\gamma}} \phi(\gamma) = \frac{c_{\gamma} b_{\gamma'}}{b_{\gamma}} \phi(\alpha_r) + b_{\gamma'} a_s \phi(\alpha_s) + b_{\gamma'} a_t \phi(\alpha_t) = \phi(\gamma'). \]
		As $\mc{B}$ is reduced, we infer $b_{\gamma} = b_{\gamma'}$. But this is a contradiction, as $\phi$ is a bijection and $\gamma \neq \gamma'$.
		
		\item $c_{\gamma} b_{\gamma'} \neq c_{\gamma'} b_{\gamma}$: Without loss of generality we may assume $c_{\gamma} b_{\gamma'} < c_{\gamma'} b_{\gamma}$. This implies $\frac{c_{\gamma} b_{\gamma'}}{b_{\gamma}} < c_{\gamma'}$. We compute the following:
		\allowdisplaybreaks
		\begin{align*}
			\frac{b_{\gamma'}}{b_{\gamma}} \phi(\gamma) + \left( c_{\gamma'} - \frac{c_{\gamma} b_{\gamma'}}{b_{\gamma}} \right) \phi(\alpha_r) &= \frac{c_{\gamma} b_{\gamma'}}{b_{\gamma}} \phi(\alpha_r) + b_{\gamma'} a_s \phi(\alpha_s) + b_{\gamma'} a_t \phi(\alpha_t) + \left( c_{\gamma'} - \frac{c_{\gamma} b_{\gamma'}}{b_{\gamma}} \right) \phi(\alpha_r) \\
			&= \phi(\gamma').
		\end{align*}
		This implies $\gamma' \in [\alpha_r, \gamma]_{\lin}^{\mc{B}} \subseteq [\alpha_r, \gamma]$. It follows from \cite[Remark~4.21]{Bischof:Construction_of_commutator_blueprints} that $\alpha_r \subsetneq \gamma \Leftrightarrow (-\gamma) \subsetneq (-\alpha_r)$. Using Remark~\ref{Remark: intervals}, we infer
		\[ (-\gamma) = (-\gamma) \cap (-\alpha_r) \subseteq (-\gamma'). \]
		This is a contradiction to the fact that $o(r_{\gamma} r_{\gamma'}) < \infty$ (see \cite[Lemma~2.8]{Bischof:Construction_of_commutator_blueprints}). \qedhere
	\end{itemize}
\end{proof}

\begin{lemma}\label{Lemma: RGD systems and roots}
	Let $K \subseteq \NN_{\geq 3}$. Then there exists an RGD~system $\mc{D}(K) = (G, (U_{\alpha})_{\alpha \in \Phi(W, S)})$ of type $(W, S)$ over $\FF_2$ with the following properties:
	\begin{itemize}
		\item $\forall n\in K: [U_{\alpha_n}, U_{\beta_n}] = U_{\gamma} U_{\gamma'}$;
		
		\item $\forall n \in \NN \setminus K: [U_{\alpha_n}, U_{\beta_n}] = \triv$.
	\end{itemize}
\end{lemma}
\begin{proof}
	We use the notation from \cite{Bischof:Construction_of_commutator_blueprints}. We have constructed in \cite[Lemma~$4.24$]{Bischof:Construction_of_commutator_blueprints} the corresponding Weyl-invariant commutator blueprint. To see the correspondence, one chooses $J_k = \{r\}$ and $L_k^r = \{2\}$ for all $k \in K$ in \cite[Lemma~$4.24$]{Bischof:Construction_of_commutator_blueprints}. These commutator blueprints are integrable by \cite[Theorem~B]{BiRGD}. This implies that RGD~systems with the corresponding commutation relations exist.
\end{proof}

\begin{theorem}\label{Theorem: 444 not linear wrt B(W,S)}
	There exist uncountably many RGD~systems of type $(4, 4, 4)$ over $\FF_2$ which are not linearizable.
\end{theorem}
\begin{proof}
	Let $K \subseteq \NN$ and let $\mc{D}(K)$ be the RGD~system of type $(W, S)$ from Theorem~\ref{Lemma: RGD systems and roots}. Let $\mc{B}$ be a reduced root basis associated with $(W, S)$. Then Lemma~\ref{Lemma: gamma/gamma' not in alphan,betan lin} implies that $\{ \gamma, \gamma' \} \not\subseteq [\alpha_n, \beta_n]_{\lin}^{\mc{B}}$. Now Lemma~\ref{Lemma: non-linear RGD system} yields that $\mc{D}(K)$ is not linear w.r.t.\ $\mc{B}$. We conclude that there are uncountably many RGD~systems of type $(4, 4, 4)$ which are not linearizable.
\end{proof}

\bibliographystyle{amsalpha}
\bibliography{references}

\end{document}